\documentclass[A4paper,
]{amsart}

\usepackage{epsf}
\usepackage{graphics}
\usepackage{graphicx}
\usepackage{amssymb}
\usepackage{amsmath}
\usepackage{mathtools}

\date{}

\theoremstyle{plain}
\newtheorem{theorem}{Theorem}
\newtheorem{corollary}[theorem]{Corollary}

\newtheorem{obs}[theorem]{Observation}
\newtheorem{lemma}[theorem]{Lemma}
\newtheorem{rem}[theorem]{Remark}
\newtheorem{question}[theorem]{Question}
\newenvironment{Rem}{\begin{rem}\rm}{\end{rem}}

\theoremstyle{definition}

\theoremstyle{remark}

\def\epsilon{\varepsilon}
\def\C{{\mathbb C}}
\def\N{{\mathbb N}}

\def\R{{\mathbb R}}
\def\P{{\mathbb P}}

\title[Asymptotics of the smooth $A_n$-realization problem]{Asymptotics of the smooth $A_n$-realization problem}
\author{Sebastian Baader, Peter Feller}
\address{Universit\"at Bern, Sidlerstrasse 5, 3012 Bern, Switzerland}
\email{sebastian.baader@unibe.ch}
\address{ETH Z\"urich, R\"amistrasse 101, 8092 Z\"urich, Switzerland}
\email{peter.feller@math.ch}
\begin{document}
\begin{abstract} We solve an asymptotic variant of a smooth version of the $A_n$-realization problem for plane curves.
As an application, we determine the cobordism distance between torus links of type $T(d,d)$ and $T(2,N)$ up to an error of at most $3d$.
We also discuss the limits of knot theoretic approaches aimed at solving the $A_n$-realization problem.
\end{abstract}


\maketitle

\section{Introduction}

The algebraic $A_n$-realization problem asks for the minimal degree~$d(n)$ of a 
polynomial $f(x,y) \in \C[x,y]$ that has an isolated singularity of type $A_n$ at the origin~\cite{GreuelLossenShustin_89_PlaneCurveswithPrescribedSing,GS08}.
The minimal degree~$d(n)$ is known to satisfy
\begin{equation}\label{eq:knownasymptotic}
\frac{7}{12} \leq \liminf_{n\to\infty}\frac{n}{d(n)^2}\leq\limsup_{n\to\infty}\frac{n}{d(n)^2} \leq \frac{3}{4}.\end{equation}
The lower bound is due to Orevkov by a concrete construction, while the upper bound results from an analysis of the signature spectrum~\cite{Orevkov_12_SomeExamples}.
 
Let $f(x,y) \in \C[x,y]$ be a square-free polynomial of degree $d$ that has an isolated singularity of type $A_n$ at the origin. In knot theoretic terms this means that, for sufficiently small $\epsilon>0$, the intersection of the algebraic curve $f^{-1}(0) \subset \C^2$ with the sphere $S^3_\epsilon\subset \C^2$ of radius $\epsilon$ around the origin is a torus link of type $T(2,n+1)$. Fixing a sufficiently small $\epsilon>0$ and adding a generic polynomial of degree $d$ with small coefficients to $f(x,y)$, we obtain a polynomial $\widetilde{f}(x,y)$ of degree~$d$ such that $\widetilde{f}^{-1}(0) \subset \C^2$ is smooth, $\widetilde{f}^{-1}(0) \cap S^3_\epsilon$ is still a torus link of type $T(2,n+1)$, and in addition the link at infinity---$\widetilde{f}^{-1}(0) \cap \partial S^3_R$ for large $R>0$---is a torus link of type $T(d,d)$; see Remark~\ref{rem:T(d,d)atinfty}.
Taking into account the classic genus-degree formula for smooth algebraic curves, one finds that the curve $\widetilde{f}^{-1}(0)$ provides a smooth connected cobordism of Euler characteristic around $n-d^2$ between the two links $T(d,d)$ and $T(2,n+1)$. In particular, if the upper bound of $\frac{3}{4}$ for the ratio $\frac{n}{d^2}$ were achieved, we would obtain a smooth cobordism of Euler characteristic around $n-d^2=-\frac{1}{4}d^2$ between the links $T(d,d)$ and $T(2,n)$. The purpose of this note is to prove that asymptotically such a smooth cobordism actually exists.

\begin{theorem}
\label{threequarters}
The maximal Euler characteristic $\chi(d)$ among all smooth connected cobordisms between the links $T(d,d)$ and $T(2,\left\lfloor\frac{3}{4}d^2\right\rfloor)$ satisfies
\[\lim_{d \to \infty} \frac{\chi(d)}{d^2}=-\frac{1}{4}.\]
\end{theorem}
Unlike the derivation of previous results on the cobordism distance between torus links~\cite{Baader_ScissorEq,BaaderFellerLewarkZentner_19,FP}, the proof of Theorem~\ref{threequarters} is not based on combinatorial braid group considerations. Instead, the main input here are examples of algebraic curves with a large number of $A_m$-singularities described by Hirano~\cite{Hirano_92}; see Section~\ref{sec:main} for the exact statement of Hirano and how it implies Theorem~\ref{threequarters}. It is surprising to the authors that algebraic curves considered over 30 years ago allow to build cobordisms that make all the observations in this note possible.

In Section~\ref{sec:asymdist}, as an application of Theorem~\ref{threequarters} and its proof
, we find that the link $T(2,\lfloor\frac{3}{4}d^2\rfloor)$ is essentially the nearest link to $T(d,d)$ among all torus links of type $T(2,N)$, in terms of the smooth cobordism distance, as defined in~\cite{Baader_ScissorEq}. 
Here is the precise result.

\begin{theorem}\label{prop:cobdist}
For all non-zero integers $d$ and all integers $N$, the maximal Euler characteristic $\chi(d,N)$ among all smooth connected cobordisms between the links $T(d,d)$ and $T(2,N)$ has value around
$-\frac{1}{4}d^2-\bigm|\mathrm{N-sign}(d)\frac{3}{4}d^2\bigm|$.
More precisely, 
\[-\frac{d^2}{4}-\left|\mathrm{sign}(d)\frac{3}{4}d^2-N\right|-4|d|
\leq\chi(d,N)\leq -\frac{d^2}{4}-\left|\mathrm{sign}(d)\frac{3}{4}d^2-N\right|+2|d|.\]
\end{theorem}

For context, we note that, of course, many values of $\chi(d,N)$ are known exactly. For example, in case $d>0>N$, the upper bound is an equality by the resolution of the local Thom conjecture~\cite{KronheimerMrowka_Gaugetheoryforemb}; see last paragraph of the proof of Theorem~\ref{prop:cobdist}. The point is that for many choices of $d$ (in particular for $d\geq 10$ and $N\geq \frac{3d^2}{4}$), the exact value of $\chi(d,N)$ remains unknown and Theorem~\ref{prop:cobdist} constitutes the first time $\chi(d,N)$ is determined up to an error that is linear in $|d|$. 
We also note that Theorem~\ref{prop:cobdist} not only recovers Theorem~\ref{threequarters} (by setting $N=\lfloor \frac{3d^2}{4}\rfloor$), but makes precise the rate of convergence in Theorem~\ref{threequarters};
see also~\eqref{eq:boundsonchi(d)}.

Next, we discuss the impact of Theorem~\ref{threequarters} on knot theoretic strategies to approach the $A_n$-realization problem.

In light of the upper bound in~\eqref{eq:knownasymptotic}, an interesting next step in the algebraic $A_n$-realization problem would be to find a constant $c<3/4$ such that
\begin{equation}\label{eq:c}
\liminf_{n\to\infty}\frac{n}{d(n)^2} \leq c.
\end{equation}
Theorem~\ref{threequarters} has a consequence, which is arguably somewhat disappointing:
an approach from knot concordance theory towards finding such a $c$ using a certain type of concordance invariants is doomed to fail. This follows from the following corollary of Theorem~\ref{threequarters} as we explain in detail in Appendix~\ref{A:hownot}.
\begin{corollary}\label{cor:aaislost}
Every real-valued $1$-Lipschitz concordance invariant $I$ with
\newline
$\lim_{m\to\infty}\frac{I(T(2,2m+1))}{g_4(T(2,2m+1))}=1$ satisfies
\[\liminf_{d\to\infty}\frac{I(T(d,d+1))}{g_4(T(d,d+1))}\geq \frac{1}{2}.\]
\end{corollary}
Examples of such invariants $I$ include many classical and recent knot invariants (when appropriately normalized), for example Trotter's signature $\sigma$, Rasmussen's $s$, Ozsv\'ath and Szab\'o's $\tau$, Ozsv\'ath, Stipsicz, and Szab\'o's $\Upsilon(t)$, and Hom and Wu's~$\nu^+$. We discuss such invariants $I$ and a proof of Corollary~\ref{cor:aaislost} in Section~\ref{sec:conchom}.

However, Theorem~\ref{threequarters} does not destroy all hope of using smooth concordance as an approach towards making progress on the $A_n$-realization problem. For context, we explain this in Appendix~\ref{sec:smoothAnrelproblem}, where we also make explicit a smooth analogue of the $A_n$-realization problem.

\subsection*{Acknowledgements} PF gratefully acknowledges support by the Swiss National Science Foundation
Grant~181199.

\section{Hirano's curves and the proof of Theorem~\ref{threequarters} }\label{sec:main}
We derive Theorem~\ref{threequarters} (and Theorem~\ref{prop:cobdist}) from the following family of examples.
\begin{lemma}\label{lem:chiexact}
For all integers $m\geq 1$, there exists a connected smooth cobordism of Euler characteristic
$-m^2-2m$  between the links
$T(2m,2m)\text{ and }
T\left(2,3m^2 \right)$.
\end{lemma}
Lemma~\ref{lem:chiexact} follows from the existence of a family of examples of projective algebraic curves due to Hirano:
for every integer $m\geq 2$, there exists an irreducible projective algebraic curve in $\C\P^2$ of degree $2m$ with exactly $N:= 3m$ singularities, all of which are of type~$A_{m-1}$~\cite[Theorem~2]{Hirano_92}; see also~\cite[Theorem~3.2]{GS08}. We explain the details.
\begin{proof}[Proof of Lemma~\ref{lem:chiexact}]
For $m=1$, there even exists a smooth connected cobordism from $T(2,2)$ to $T(2,3)$ with Euler characteristic $-1$.
Hence we consider the case 
of an integer $m\geq 2$. 
We 
consider the projective algebraic curve $C\subset \C\P^2$ of degree $d\coloneqq 2m$ given by the irreducible homogeneous polynomial
\[ F= \left(x^{m}+y^{m}+z^{m}\right)^2-4\left(x^{m}y^{m}+y^{m}z^{m}+z^{m}x^{m}\right)\in\C[x,y,z],\] which was used by Hirano to prove~\cite[Theorem~2]{Hirano_92}. A calculation reveals that $C$ has 
$N:= 3m$ singularities, all of which are of type~$A_{m-1}$.

We obtain an affine algebraic curve in $\C^2$ with $N$-many $A_{m-1}$ singularities by removing a generic projective line from $\C\P^2$. We explain this in more detail in the rest of this paragraph.
Pick a projective line $L$---a subvariety $L\subset \C\P$ defined by $ax+by+cz=0$ for some
$[a:b:c]\in\C\P^2$---that intersects $C$ transversally
(i.e.~if $p\in C\cap L$, then $p$ is a non-singular point of $C$ and the tangent spaces of $C$ and $L$ at $p$ span the entire tangent space of $\C\P^2$ at $p$).
Note that by B\'ezout's theorem $C\cap L$ consists $d$ points. We pick a linear transformation $A\in \mathrm{Gl}(3,\C)$ that maps $L$ to the line at infinity, which is defined by $z=0$ and denoted by $\C\P^1$. Now consider $\widetilde{F}=F\circ A^{-1}$,
which is a homogeneous polynomial defining the curve $\widetilde{C}$ obtained from applying $A$ to $C$,
define $\widetilde{f}\coloneqq \widetilde{F}(x,y,1)$, and take the desired affine algebraic curve in $\C^2=\C\P^2\setminus \C\P^1$ to be $\widetilde{f}^{-1}(0)$. 

The fact that $\widetilde{C}$ intersects $\C\P^1$ transversally,
implies that the link at infinity of $\widetilde{f}^{-1}(0)$ is the torus link $T(d,d)$, i.e.~$S^3_R\cap \widetilde{f}^{-1}(0)$ has link type $T(d,d)$ for $R>0$ large enough.
In order to see this, consider a closed regular neighborhood $\nu(\C\P^1)$ of $\C\P^1\subset\C\P^2$. To be concrete, take $\nu(\C\P^1)$ to be the complement of $\mathrm{int}(B_R)\subset\C^2\subset\C\P^2$ for some large $R>0$. Such a neighborhood is diffeomorphic (via some $\phi$) to the total space $E$ of the once-twisted $D^2$-bundle over $S^2\cong \C\P^1$---the $D^2$-bundle $\pi\colon E\to S^2$ with $ \partial E=S^3$.
Choosing $\nu(\C\P^1)$ smaller if needed (that is increasing $R$), by transversality we can arrange for the diffeomorphism $\phi\colon \nu(\C\P^1)\to E$ to map $\nu(\C\P^1)\cap \widetilde{C}$ to $d$ fibers $\pi^{-1}(p_1),\pi^{-1}(p_2),\cdots,\pi^{-1}(p_d)$ for points $p_1,\cdots,p_d\in S^2$.
Since $\pi|_{\partial E}\colon \partial E\to S^2$ is the Hopf fibration (since there is only one $S^1$-bundle over $S^2$ with total space $S^3$), $T\coloneqq\pi|_{\partial E}^{-1}\{p_1,\cdots,p_d\}\subset \partial E=S^3$ is a $T(d,d)$ torus link.
Recalling $\partial(\nu(\C\P^1))=S^3_R$, we see that $\phi$ yields a diffeomorphism of pairs between 
$(S^3_R,S^3_R\cap \widetilde{f}^{-1}(0))=(S^3_R,\phi^{-1}(T))$ and $(S^3, T)$. Hence, as desired, the link at infinity is the torus link $T(d,d)$.

We now use $\widetilde{f}^{-1}(0)$ to find the desired cobordism. For this
let $s_1,\cdots, s_N$ denote the singular points of $\widetilde{f}^{-1}(0)$ and choose $\epsilon$ such that $S^3_{\epsilon,s_k}$---the sphere of radius epsilon around $s_k$---intersects $\widetilde{f}^{-1}(0)$ transversally in $T(2,m)$. Let $W:= B_R^4\setminus(\bigcup_{1\leq k\leq N}\mathrm{int}(B^4_k))$, where $B^4_R$ denotes the close ball with boundary $S^4_R$ and $\mathrm{int}(B^4_k)$ denotes the open ball with boundary $S^3_{\epsilon,s_k}$. Towards finding our cobordism (a surface $F$ in $S^3\times [0,1]$),
we modify $W$ to be diffeomorphic to $S^3\times [0,1]$ by tubing the small boundary spheres together and track what happens to $\widetilde{f}^{-1}(0)$.
For this, we pick $N-1$ pairwise disjoint embedded closed arcs $a_k\subset W\cap\widetilde{f}^{-1}(0)$ that start and end on
$S\coloneqq\bigcup_{1\leq k\leq N} S^3_{\epsilon,s_k}$
such that $S\cup \bigcup_{1\leq k\leq N-1} a_k$ is connected.
We take $\nu(a_k)$ to be a small closed tubular neighborhood of $a_k$ such that its boundary intersects $\widetilde{f}^{-1}(0)$ transversally and the pair
$(\nu(a_k), \widetilde{f}^{-1}(0)\cap\nu(a_k))$ is diffeomorphic to $(B^3\times[0,1],B^2\times[0,1])$, where $B^3$ denotes the unit ball in $\R^3$ and $B^2{\coloneqq}B^3\cap \R^2\times\{0\}$.
We set $X$ to be the closure in $\C^2$ of $W\setminus\bigcup_{1\leq k\leq N} \nu(a_k)$. After smoothing corners, we pick a diffeomorphism from $X$ to $S^3\times [0,1]$ and let $F\subset  S^3\times [0,1]$ denote the image of $X\cap\widetilde{f}^{-1}(0)$.

Next, we determine the Euler characteristic of $F$ and along the way observe that it is connected.
We discuss the case that $m$ is odd. A similar calculation works when $m$ is even and yields the same result. 
We first note that $C$ (as a topologial surface) is connected, as is the case for all irreducible projective curves in $\C\P^2$, and has genus $(d-1)(d-2)/2-N\frac{m-1}{2}$. The latter can for example be seen by noting that a small generic deformation of $C$ is a smooth algebraic curve of degree $d$, which has genus $(d-1)(d-2)/2$, where each of the $N$-many $A_{m-1}$-singularities contributes $\frac{m-1}{2}$ to the genus. 
$\widetilde{f}^{-1}(0)$ and $\widetilde{f}^{-1}(0)\cap W$ have the same genus as $C$ (since they are obtained by removing discs), and $\widetilde{f}^{-1}(0)\cap X\cong F$ also has the same genus as $C$ and is also connected since it is obtained from $\widetilde{f}^{-1}(0)\cap W$ by removing neighborhoods of embedded arcs that connect different boundary components. Thus, we find
\[\chi(F)=2-2\frac{(d-1)(d-2)-N(m-1)}{2}-(d+1)
=-(d-1)(d-1)+N(m-1),\] because $F$ is connected and has $d+1$ boundary components.

By construction, $F$ is a connected smooth cobordism between $T(d,d)$ and a knot $K$ given as the connected sum of $N$-many $T(2,m)$ torus links. Take $F'$ to be a cobordism given by $N-1$ one-handles between $K$ and $T(2,Nm)$; in particular, $F'$ is connected and $\chi(F')=-N+1$.

Composing the two cobordisms $F$ and $F'$, we find a connected smooth cobordism from $T(d,d)$ to $T(2,Nm)$
with Euler characteristic
\[-(d-1)(d-1)+N(m-1)-N+1=-(d-1)(d-1)+(m-2)N+1=-m^2-2m.\qedhere\]
\end{proof}

\begin{proof}[Proof of Theorem~\ref{threequarters}]
We fix an integer $d\geq 2$ and consider the largest integer $m$ such that $d\geq D\coloneqq 2m$. In other words, $D=d$ if $d$ is even and $D=d-1$ if $d$ is odd. Let $F$ be a smooth connected cobordism from $T(d,d)$ to $T(D,D)$ with
\[\chi(F)=(D-1)^2-(d-1)^2=\left\{\begin{array}{cc}0&\text{ if $d$ is even}\\
-4m+1&\text{ if $d$ is odd}\end{array}\right..\]
Let $F'$ be a smooth connected cobordism from $T\left(2,3m^2\right)$ to $T\left(2,\lfloor \frac{3}{4}d^2\rfloor\right)$ with
\[\chi(F')=-3m^2+\lfloor 3d^2/4\rfloor=\left\{\begin{array}{cc}0&\text{ if $d$ is even}\\
-3m&\text{ if $d$ is odd}\end{array}\right..\] 
Hence, by composing the following three cobordisms
\begin{itemize}
\item $F$ from $T(d,d)$ to $T(D,D)$,
\item a cobordism $H$ from $T(D,D)$ to $T(2,3m^2)$ as guaranteed to exist by Lemma~\ref{lem:chiexact},
\item and $F'$ from $T(2,3m^2)$ to $T(2,\lfloor \frac{3}{4}d^2\rfloor)$,
\end{itemize} we find a connected cobordism $G$ between $T(d,d)$ and $T(2,\left\lfloor\frac{3}{4}d^2\right\rfloor)$ with
\begin{align*}
\chi(G)&= \chi(F)+\chi(H)+\chi(F')
\\&=\left\{\begin{array}{cc}0-m^2-2m+0=-m^2-2m&\text{ if $d$ is even}\\
-4m+1-m^2-2m-3m=-m^2-9m+1&\text{ if $d$ is odd}\end{array}\right.
\end{align*}
Therefore, we have
\[\chi(d)\geq \left\{\begin{array}{cc}-m^2-2m=-\frac{d^2}{4}-d&\text{ if $d$ is even}\\-m^2-9m+1=-\frac{d^2}{4}-4d+5+\tfrac{1}{4}&\text{ if $d$ is odd}\end{array}\right.\geq-\frac{d^2}{4}-4d.\]

To find an upper bound on $\chi(d)$, we employ Murasugi's signature obstruction on cobordism distance~\cite{Murasugi_OnACertainNumericalInvariant}.
Using
\begin{equation}\label{eq:sigfortorusknots}\sigma(T(d,d))=-\left\lfloor\tfrac{d^2-1}{2}\right\rfloor\text{ and }\sigma\left(T\left(2,k\right)\right)=-k+1\quad\text{\cite[Theorem~5.2]{GLM}},\end{equation}
for all positive integers $d$ and $k$, we find
\[\chi(d)\overset{\text{\cite{Murasugi_OnACertainNumericalInvariant}}}{\leq}
\sigma\left(T\left(2,\left\lfloor\tfrac{3}{4}d^2\right\rfloor\right)\right)-\sigma(T(d,d))\overset{\text{\eqref{eq:sigfortorusknots}}}{=}\left\lfloor\tfrac{d^2-1}{2}\right\rfloor-\left\lfloor\tfrac{3}{4}d^2\right\rfloor+1\leq -\left\lfloor\tfrac{d^2}{4}\right\rfloor+1.\]

In conclusion, we have shown
\begin{equation}\label{eq:boundsonchi(d)}-\frac{d^2}{4}-4d\leq\chi(d)\leq-\left\lfloor\frac{d^2}{4}\right\rfloor+1;\end{equation}
in particular, this establishes $\lim_{d \to \infty} \frac{\chi(d)}{d^2}=-\frac{1}{4}$.
\end{proof}

We end this section with a remark about rearranging curves in $\C^2$ such that their link at infinity is $T(d,d)$ without changing the singularity at the origin, which we have used in the first paragraph of the introduction. The argument is very similar to the one from the second paragraph of the proof of Lemma~\ref{lem:chiexact}.
\begin{Rem}\label{rem:T(d,d)atinfty}
Let $C\subset\C^2$ be a reduced\footnote{reduced simply amounts to the defining polynomial $f$ being square-free} algebraic curve of degree $d$ and fix some $p\in C$. We claim by adding a small degree $d$ polynomial to a defining square-free polynomial $f$ of $C$, we can change $C$ to a reduced algebraic curve $\widetilde{C}$ with the same singularity at $p$ such that the link at infinity of $\widetilde{C}$ is $T(d,d)$.

Indeed, this can be done similarly as argued in the second paragraph of the above proof. 
Considering the closure in $\C\P^2$ (by homogenizing to a 3-variable polynomial $F$ of degree $d$ with $F(x,y,1)=f(x,y)$), chose a generic projective line, and then composing the defining equation $F$ with a linear transformation of $\C\P^2$ that maps this line to the line at infinity $\{[x:y:z]\mid z=1\}$ and fixes the $p$, we find a new defining equation $\widetilde{F}$ such that setting $\widetilde{f}(x,y)=\widetilde{F}(x,y,1)$ yields the defining equation of an algebraic curve $\widetilde{C}$ as desired. To guarantee that the coefficients between $f$ and $\widetilde{f}$ vary little, choose the generic projective line to be given by an equation $ax+by+cz=0$ with $a$ and $b$ close to $0$ and $c$ close to $1$ and choose the linear transformation close to the identity (say as an element in $\mathrm{Gl}_3(\C)$). 

In the second paragraph of the introduction, we further wanted a smooth algebraic curve. For this we note that adding a generic linear polynomial arranges that $\widetilde{f}^{-1}(0)$ is smooth without changing the link at infinity. Choosing the coefficients of said linear polynomial small (compared to an $\epsilon$ with $S^3_{\epsilon'}\pitchfork \widetilde{f}^{-1}(0)$ being a torus link of type $T(2,n+1)$ for all $0<\epsilon'\leq\epsilon$) assures that $\widetilde{f}^{-1}(0) \cap S^3_\epsilon$ remains a torus link of type $T(2,n+1)$.
\end{Rem}  

\section{Cobordism distance between torus links of type $T(d,d)$ and $T(2,N)$}\label{sec:asymdist}
Theorem~\ref{threequarters} and its proof combined with the knot signature obstruction for cobordisms and the resolution of the local Thom conjecture allows to determine the smooth cobordism distance between the link $T(d,d)$ and all torus links  $T(2,N)$ up to an error of at most $3d$. This is the content of Theorem~\ref{prop:cobdist}, which we now prove.

\begin{proof}[Proof of Theorem~\ref{prop:cobdist}]
Without loss of generality, take $d$ to be positive. In fact, we consider the case when $d\geq 2$, since the case $d=1$ is immediate from the local Thom conjecture; compare~\eqref{eq:locthom} below.

We first discuss the case $N\geq0$. 
Using a cobordism between $T(d,d)$ and $T(2,\lfloor \frac{3}{4}d^2\rfloor)$ that realizes $\chi(d)$ and composing it with a connected cobordism between $T(2,\lfloor \frac{3}{4}d^2\rfloor)$ to $T(2,N)$ of Euler characteristic $-\left|\lfloor \frac{3}{4}d^2\rfloor-N\right|$, yields $-\left|\lfloor \frac{3}{4}d^2\rfloor-N\right|+\chi(d)\leq\chi(d,N)$. Combined with
\[-\left|\lfloor \tfrac{3}{4}d^2\rfloor-N\right|-\frac{d^2}{4}-4d
\overset{\text{\eqref{eq:boundsonchi(d)}}}{\leq}-\left|\lfloor \tfrac{3}{4}d^2\rfloor-N\right|+\chi(d),\]
we find the desired lower bound
\[-\frac{d^2}{4}-\left|\frac{3}{4}d^2-N\right|-4d
\leq\chi(d,N).\]

For the upper bound, in case $N\geq \frac{3}{4}d^2$, we recall the signature bound
\[\chi(d,N)\overset{\text{\cite{Murasugi_OnACertainNumericalInvariant}}}{\leq} -\sigma(T(d,d))+\sigma(T(2,N))\overset{\eqref{eq:sigfortorusknots}}{=} \left\lfloor\tfrac{d^2-1}{2}\right\rfloor-N+1{\leq} -\left|-\tfrac{d^2}{2}+N\right|+\tfrac{1}{2}\]
and apply
$\left|-\frac{d^2}{2}+N\right|=\left|\frac{d^2}{4}+(-\frac{3}{4}d^2+N)\right|= \frac{d^2}{4}+\left|-\frac{3}{4}d^2+N\right|$ to find
\[\chi(d,N)\leq -\frac{d^2}{4}-\left|-\frac{3}{4}d^2+N\right|+\frac{1}{2},\] as desired.
If instead, $0\leq N< \frac{3}{4}d^2$, we use the following triangle inequality for the cobordism distance
$\chi(d,N)\leq \chi(d,1) - \chi(1,N)$ in combination with the following consequence of the local Thom conjecture~\cite[Corollary~1.3]{KronheimerMrowka_Gaugetheoryforemb}:
\begin{equation}\label{eq:locthom}
\chi(d,1)=-(|d|-1)^2\text{ and } \chi(1,N)=-\left||N|-1\right|,\end{equation} for all integers $N$ and non-zero integers $d$.
We find
\begin{align*}
\chi(d,N)&\leq-(d-1)^2+(N-1)
\\&= -d^2+N+2d-2
\\&=-\frac{d^2}{4}-\left|\frac{3}{4}d^2-N\right|+2d-2,\end{align*} where we combined the triangle inequality and~\eqref{eq:locthom} to see the inequality.

Finally, if $N\leq -1$, then
\[\chi(d,N)= \chi(d,1)+\chi(1,N)\overset{\text{\eqref{eq:locthom}}}{=}-(d-1)^2+N+1
=-\tfrac{d^2}{4}-\left|\tfrac{3}{4}d^2-N\right|+2d,\] where the first equality is a consequence of the local Thom conjecture~\cite[Corollary~1.3]{KronheimerMrowka_Gaugetheoryforemb}.
\end{proof}

\section{$1$-Lipschitz concordance invariants}\label{sec:conchom}
We call a real-valued knot invariant $I\colon \mathfrak{K}nots\to \R$ a \emph{1-Lipschitz concordance invariant} if $|I(K)-I(J)|\leq g_4(J\# -K)$ for all $K, J\in \mathfrak{K}nots$, where
$\mathfrak{K}nots$ denotes the set of isotopy classes of knots and $-K$ denotes the reverse of the mirror of $K$.

Most classically, Trotter's signature $-\sigma/2$
is an example~\cite{Trotter_62_HomologywithApptoKnotTheory,Murasugi_OnACertainNumericalInvariant},
but also Ozsv\'ath and Szab\'o's $\tau$~\cite{OzsvathSzabo_03_KFHandthefourballgenus} and $-s/2$~\cite{rasmussen_sInv} (and more generally all slice-torus invariants), and Ozsv\'ath, Stipsicz, and Szab\'o's $-\Upsilon(t)/t$~\cite{OSS_2014}. All of these are also additive under connected sum. A none-additive  example is Hom and Wu's $\nu^+$~\cite{Hom_2016}.
These examples of $1$-Lipschitz concordances invariants satisfy $|I(T_{2,2m+1})|=m$ for $m\in\N$. For such $I$, as a consequence of Theorem~\ref{threequarters}, we find Corollary~\ref{cor:aaislost}, which can be paraphrased to say that $|I(T(d,d+1))|$ is at least half of the genus of $T(d,d+1)$ asymptotically for large $d$. This might be of independent interest, but for us this is actually a negative result since it shows that a certain approach towards making progress on the $A_n$-realization problem can not work; see Appendix~\ref{A:hownot}, where we make this statement precise.

\begin{proof}[Proof of Corollary~\ref{cor:aaislost}]
Fix a positive integer $d$ and write $N\coloneqq \left\lfloor\frac{3}{4}d^2\right\rfloor$. By composing a connected cobordism of Euler characteristic $\chi(d)$ between $T(d,d)$ and $T(2,N)$ with a connected cobordism of Euler characteristic $1-d$ between $T(d,d+1)$ and $T(d,d)$, we find a connected cobordism of Euler characteristic $\chi(d)-(d-1)$ between $T(d,d+1)$ and $T(2,N)$. If $N$ is odd, we take $F$ to be this cobordism and write $K=T(2,N)$, if not we take $F$ to be a connected cobordism between $T(d,d+1)$ and $T(2,N+1)$ of Euler characteristic $\chi(d)-d$ and write $K=T(2,N+1)$. In both cases, $F$ has genus
$\lceil\frac{-\chi(d)+(d-1)}{2}\rceil$; hence, $g_4(T(d,d+1)\# -K\leq \lceil\frac{-\chi(d)+(d-1)}{2}\rceil$. We complete the proof by the following calculation, which uses Theorem~\ref{threequarters} for the first equality and the assumption $\lim_{m\to\infty} \frac{I(T(2,2m+1))}{g_4(T(2,2m+1)}=1$ for  the last equality:
\begin{align*}\frac{1}{4}
&=\lim_{d\to\infty}
\frac{-\chi(d)}{d^2}
=\lim_{d\to\infty}
\frac{\lceil\frac{-\chi(d)+(d-1)}{2}\rceil}
{d^2/2}
=\lim_{d\to\infty}
\frac{\lceil\frac{-\chi(d)+(d-1)}{2}\rceil}
{g_4(T(d,d+1))}
\\&=\lim_{d\to\infty}
\frac{\lceil\frac{-\chi(d)+(d-1)}{2}\rceil}
{g_4(T(d,d+1))}
=\liminf_{d\to\infty}
\frac{\lceil\frac{-\chi(d)+(d-1)}{2}\rceil}
{g_4(T(d,d+1))}\\
&\geq \liminf_{d\to\infty}\frac{I(K)-I(T(d,d+1))}{g_4(T(d,d+1))}\\
&\geq \liminf_{d\to\infty}\frac{I(K)}{g_4(T(d,d+1))}-\liminf_{d\to\infty}\frac{I(T(d,d+1))}{g_4(T(d,d+1))}\\
&= \liminf_{d\to\infty}\frac{I(K)}{d^2/2}-\liminf_{d\to\infty}\frac{I(T(d,d+1))}{g_4(T(d,d+1))}\\
&= \frac{3}{4}-\liminf_{d\to\infty}\frac{I(T(d,d+1))}{g_4(T(d,d+1))}.\qedhere
\end{align*}
\end{proof}

\appendix
\section{Context: a smooth analogue of the $A_n$-realization problem and limitations of Theorem~\ref{threequarters}}\label{sec:smoothAnrelproblem}
Let us be exact in determining the Euler characteristic of the cobordism provided by $\widetilde{f}$ between $T(d,d)$ and $T(2,n+1)$ from the second paragraph of the introduction, which we earlier found to be around $n-d^2$. Its Euler characteristic is $n-(d-1)^2$, as we explain in the rest of this paragraph. Take $C\subset \C\P^2$ to be the closure of $\widetilde{f}^{-1}(0)$, i.e.~the projective algebraic curve given by the homogenization of $\widetilde{f}$. We note that $C$ is a smooth curve (this follows, since all points of $C$ in $\C^2\subseteq \C\P^2$ are non-singular by the choice of $\widetilde{f}$ and from the fact that the link at infinity is $T(d,d)$ we find that $C$ has $d$ non-singular points on $\C\P^1\coloneqq \C\P^2\setminus \C^2$) of degree $d$; hence, it is a closed genus $(d-1)(d-2)/2$ surface and thus $\widetilde{f}^{-1}(0)$ is a $d$-times punctured genus $(d-1)(d-2)/2$ surface. The link $T(2,n+1)$ separates $\widetilde{f}^{-1}(0)$ into two pieces, one of which (the bounded one) is diffeomorphic to the Milnor fiber $F$ of the $A_n$-singularity, i.e.~a connected surface with first Betti number (aka its Milnor number) equal to $n$. Hence, the Euler characteristic of the cobordism is
\[\chi\left(\widetilde{f}^{-1}(0)\right)-\chi(F)=(-(d-1)(d-2)+2-d)-(-n+1)=n-(d-1)^2.\]

Motivated by the above calculation, we let $d_{\mathrm{sm}}(n)$ denote the smallest integer such that there exists a connected smooth cobordism of Euler characteristic $n-(d_{\mathrm{sm}}(n)-1)^2$ between $T(d_{\mathrm{sm}}(n),d_{\mathrm{sm}}(n))$ and $T(2,n+1)$. Equivalently, invoking the resolution of the local Thom conjecture, $d_{\mathrm{sm}}(n)$ is the smallest integer among the positive integers $d$ such that there exists a $\chi$-maximizing smooth connected cobordism $C\subset S^3\times [-1,1]$ between $T(d,d)$ and the unknot $U$ with $S^3\times \{0\}\pitchfork C=T(2,n+1)$. In other words, $d_{\mathrm{sm}}(n)$ is the smallest integer among the $d$ with
\[d_\mathrm{cob}(T(d,d), U)=d_\mathrm{cob}(T(d,d), T(2,n+1))+d_\mathrm{cob}(T(2,n+1), U),\] where 
$d_\mathrm{cob}$ denotes the cobordism distance between links; see also~\cite[Obs.~5]{Feller_15_MinCobBetweenTorusknots}.

The problem of determining $d_{\mathrm{sm}}(n)$ can be understood as a smooth analogue of the $A_n$-realization problem. Certainly, by the above calculation, one has $d_{\mathrm{sm}}(n)\leq d(n)$, but it is even conceivable that the following question has a positive answer:
Is $d_{\mathrm{sm}}(n)=d(n)$ for all $n\in \N$?
This question appears to be folklore among a some knot theorists, but no answer is in sight. In any case, since $d_{\mathrm{sm}}(n)\leq d(n)$, every constant $c$ with
\begin{equation}\label{eq:csmooth}
\liminf_{n\to\infty}\frac{n}{d_{\mathrm{sm}}(n)^2} \leq c,
\end{equation}
also satisfies~\eqref{eq:c}. Therefore, a positive answer to the following smooth concordance question would constitute progress on the algebraic $A_n$-realization problem.
\begin{question}\label{q:assmoothAn}
Does there exists a constant $c<\frac{3}{4}$ that satisfies
\[\liminf_{n\to\infty}\frac{n}{d_{\mathrm{sm}}(n)^2} \leq c?\]
\end{question}

While we suspect that the answer is no, in fact, we suspect $\limsup_{n\to\infty}\frac{n}{d_{\mathrm{sm}}(n)^2}=\frac{3}{4}$, we do not know. In particular, we note that Theorem~\ref{threequarters} and its proof do \emph{not} directly provide insight into Question~\ref{q:assmoothAn} since the cobordisms between $T(d,d)$ and $T(2,\lfloor\frac{3d^2}{4}\rfloor)$ we use have Euler characteristic strictly less than
$\lfloor\frac{3d^2}{4}\rfloor-(d-1)^2$. However, Theorem~\ref{threequarters} does show that certain asymptotic values of certain knot invariants cannot be used to answer Question~\ref{q:assmoothAn}. We explain the latter in the next Appendix.

For context, we also note that the for $d_{\mathrm{sm}}(n)$ in place of $d(n)$ the upper bounds from~\eqref{eq:knownasymptotic} also hold, while the lower bound is in fact better; see~\cite[Theorem~3.13]{Orevkov_12_SomeExamples}:
 \begin{align}\label{eq:knownasymptoticsmooth}
\frac{2}{3}
\leq\liminf_{n\to\infty}\frac{n}{d_{\mathrm{sm}}(n)^2}
\leq\limsup_{n\to\infty}\frac{n}{d_{\mathrm{sm}}(n)^2} \leq \frac{3}{4}
.
\end{align}
\section{How not to resolve the $A_n$-realization problem}\label{A:hownot}

One may wonder what kind of invariants could help to answer Question~\ref{q:assmoothAn}.

The following observation provides upper bounds on the asymptotic value of $\frac{n}{d_{\mathrm{sm}}(n)}$ (and hence $\frac{n}{d(n)^2}$).

\begin{obs}\label{obs}
Let $I\colon\mathfrak{K}nots\to \R$ be a 1-Lipschitz concordance invariant with
$\lim_{m\to\infty}\frac{I(T(2,2m+1))}{g_4(T(2,2m+1))}=1$.
Setting
\[c'{\coloneqq} \liminf_{d\to\infty}\frac{I(T(d,d+1))}{g_4(T(d,d+1))}\text{ and }
c''{\coloneqq} \limsup_{d\to\infty}\frac{I(T(d,d+1))}{g_4(T(d,d+1))},\]
we find
\[\liminf_{n\to\infty}\frac{n}{d_{\mathrm{sm}}(n)^2}\leq \frac{1+c'}{2}\text{ and }\limsup_{n\to\infty}\frac{n}{d_{\mathrm{sm}}(n)^2}\leq \frac{1+c''}{2}.\]
\end{obs}
We note that in the assumption of Observation~\ref{obs} and similarly in Corollary~\ref{cor:aaislost}, the limit could be replaced with $\liminf$ since $\frac{I(K)}{g_4(K)}\leq 1$.

At first sight Observation~\ref{obs} looks like a promising approach towards answering Question~\ref{q:assmoothAn}. For example, the upper bound in~\eqref{eq:knownasymptotic} and~\eqref{eq:knownasymptoticsmooth} immediately follows using $I=-\sigma/2$ since $\lim_{d\to\infty}\frac{-\sigma(T(d,d+1))/2}{g_4(T(d,d+1))}=\frac{1}{2}$. In fact, this upper bound via the signature and Observation~\ref{obs} is essentially how the upper bound via the signature spectrum (mentioned in the first paragraph of the introduction) works.

The bad news is that, by Corollary~\ref{cor:aaislost},
for every $1$-Lipschitz concordance invariant $I\colon\mathfrak{K}nots\to \R$ with
$\lim_{m\to\infty} \frac{I(T(2,2m+1))}{g_4(T(2,2m+1))}=1$,
we have
\[\liminf_{d\to\infty}I(T_{d,d+1})/g_4(T(d,d+1))\geq \frac{1}{2}.\] 
This means, there is no $I$ that can be plugged into Observation~\ref{obs} to improve the upper bound of $\frac{3}{4}$ on any of the quantities
\[\liminf_{n\to\infty}\frac{n}{d(n)^2}, \liminf_{n\to\infty}\frac{n}{d_{\mathrm{sm}}(n)^2}, \limsup_{n\to\infty}\frac{n}{d(n)^2}, \text{and }\limsup_{n\to\infty}\frac{n}{d_{\mathrm{sm}}(n)^2}.\]

It remains to prove Observation~\ref{obs}.

\begin{proof}[Proof of Observation~\ref{obs}]
We discuss only the inequality involving $\liminf$, the other follows by a similar argument.
Fix integers $n,d>0$, where we take $n$ to be even. Assume that there exists a connected cobordism of Euler characteristic $n-(d-1)^2$ between $T(d,d)$ and $T(2,n+1)$.
Then there exists a connected cobordism of Euler characteristic $n-(d-1)d$ between  $T(d,d+1)$ and $T(2,n+1)$. This cobordism has genus $\frac{(d-1)d-n}{2}$; hence,
$\frac{(d-1)d}{2}-\frac{n}{2}\geq -I(T(d,d+1))+I(T(2,n+1))
,$ and we have
\begin{equation}\label{eq:boundI/g4} 1+\frac{I(T(d,d+1))}{g_4(T(d,d+1))}\geq \frac{\frac{n}{2}+I(T(2,n+1))}{\frac{(d-1)d}{2}}=\frac{n+2I(T(2,n+1))}{{(d-1)d}}.\end{equation}

Taking $\liminf$, we find
\begin{align*}1+c'
&\overset{\text{\phantom{\eqref{eq:boundI/g4}}}}{=}1+\liminf_{d\to\infty}\frac{I(T(d,d+1))}{g_4(T(d,d+1))}
&\overset{\text{\phantom{\eqref{eq:boundI/g4}}}}{=} & \;1+\liminf_{n\to\infty\mid n\text{ odd}}\frac{I(T(d_{\mathrm{sm}}(n),d_{\mathrm{sm}}(n)+1))}{g_4(T(d_{\mathrm{sm}}(n),d_{\mathrm{sm}}(n)+1))}
\\
&\overset{\text{\eqref{eq:boundI/g4}}}{\geq} \liminf_{n\to\infty\mid n\text{ odd}}\frac{n+2I(T(2,n+1)}{(d_{\mathrm{sm}}(n)-1)d_{\mathrm{sm}}(n)}
&\overset{\text{\phantom{\eqref{eq:boundI/g4}}}}{=}&\; \liminf_{n\to\infty}\frac{n+2I(T(2,n+1))}{(d_{\mathrm{sm}}(n)-1)d_{\mathrm{sm}}(n)}
\\
&\overset{\text{\phantom{\eqref{eq:boundI/g4}}}}{=} \liminf_{n\to\infty}\frac{n+2I(T(2,n+1))}{d_{\mathrm{sm}}(n)^2}
&\overset{\text{\phantom{\eqref{eq:boundI/g4}}}}{=}&\; \liminf_{n\to\infty}\frac{2n}{d_{\mathrm{sm}}(n)^2},
\end{align*} which completes the proof. We comment on why the equalities hold.
The first one is by definition of $c'$.
For the second one, $\leq$ is clear, but not needed. We argue for $\geq$.
By~\eqref{eq:knownasymptoticsmooth} we know that for every large $d$ there exists an even $n$ with $|d_\mathrm{sm}(n)-d|\leq 2\sqrt{d}$. Picking  $d'=d_\mathrm{sm}(n)$ for some such $n$,
we have
\[|I(T(d,d+1))-I(T(d,d+1))|, |g_4(T(d,d+1))-g_4(T(d,d+1))|\leq O((d-d')^2)\leq O(d),\]
and, since $g_4(T(d,d+1))$ grows quadratically in $d$, $\leq$ (in fact $=$) follows. 
The third to last equality follows by a similar argument using that every $d_{\mathrm{sm}}(n)$ for $n$ odd is linearly (in $d_{\mathrm{sm}}(n)$) close to $d_{\mathrm{sm}}(n\pm 1)$. The second to last equality is clear since the two denominators are only $d_{\mathrm{sm}}(n)$ apart but both grow quadratically. Finally, the last equation follows from
$\lim_{m\to\infty}\frac{I(T(2,2m+1))}{g_4(T(2,2m+1))}=1$.
\end{proof}


\bibliographystyle{alpha}
\bibliography{peterbib}

\end{document}